\documentclass[12 pt,psamsfonts,leqno,oneside,letterpaper,left=1.25truein, top=1truein, margin=2.5cm ]{article}
\usepackage[letterpaper,left=1.75truein, top=1truein]{geometry}
\usepackage{amssymb,amsmath, amsthm, amssymb, amscd,graphics}
\usepackage{epsfig}
\usepackage{color}

\usepackage[colorlinks,linkcolor=blue,citecolor=blue]{hyperref}
\usepackage{indentfirst}
\usepackage{geometry}
\def\bysame{$\underline{\hskip.5truein}$}

\newtheorem{theorem}{Theorem}[section]
\newtheorem{lemma}[theorem]{Lemma}
\newtheorem{corollary}[theorem]{Corollary}
\newtheorem{conjecture}[theorem]{Conjecture}
\newtheorem{definition}[theorem]{Definition}
\newtheorem{proposition}[theorem]{Proposition}
\newtheorem{remark}[theorem]{Remark}
\newtheorem{question}[theorem]{Question}

\def\Isom{\mbox{\rm{Isom}}}

\def\orb{{\rm{orb}}}

\def\Z{\mathbb{Z}}
\def\O{\mathcal{O}}
\def\TO{\widetilde{\mathcal{O}}}

\title{Knot complements, hidden symmetries and reflection orbifolds} 
\author{M. Boileau\thanks{Partially supported 
by Institut Universitaire de France.}, S. Boyer\thanks{Partially supported by NSERC grant OGP0009446.}, R. Cebanu\thanks{Partially supported from the contract PN-II-ID-PCE 1188 265/2009.} \ \& G. S.  Walsh\thanks{Partially supported by NSF grant 1207644.}}
\begin{document}
\maketitle

\abstract{In this article we examine the conjecture of Neumann and Reid that the only hyperbolic knots in the $3$-sphere which admit hidden symmetries are the figure-eight knot and the two dodecahedral knots. Knots whose complements cover hyperbolic reflection orbifolds admit hidden symmetries, and we verify the Neumann-Reid conjecture for knots which cover small hyperbolic reflection orbifolds. We also show that a reflection orbifold covered by the complement of an AP knot is necessarily small.  Thus when $K$ is an AP knot,  the complement of $K$ covers a reflection orbifold exactly when $K$ is either the figure-eight knot or one of the dodecahedral knots.}

\section{Introduction} 
\noindent Commensurability is the equivalence relation on families of spaces or orbifolds determined by the property of sharing a finite degree cover. It arises in a number of areas of interest and is connected with other important concepts. For instance, two complete, finite volume, non-compact hyperbolic orbifolds are commensurable if and only if their fundamental groups are quasi-isometric \cite{Sch}. 

In this paper we consider commensurability relations between the complements of hyperbolic knots in the $3$-sphere. For convenience we say that two such knots  are commensurable if their complements have this property. There are two conditions which play a pivotal role here, though more for their rarity than their regularity. The first is {\it arithmeticity}, a commensurability invariant which, among knots, is only satisfied by the figure-eight \cite{Re1}. In particular, the figure-eight is the unique knot in its commensurability class. A basic result due to Margulis \cite{Mar} states that the commensurability class of a non-arithmetic finite-volume hyperbolic orbifold has a minimal element. In other words, there is a hyperbolic orbifold covered by every member of the class. This holds for the complement of any hyperbolic knot $K$ other than the figure eight. For such knots we use 
$\O_{full}(K)$ to denote the minimal orbifold in the commensurability class of $S^3 \setminus K$ and $\O_{min}(K)$ to denote the minimal orientable orbifold in the commensurability class of 
$S^3 \setminus K$. These orbifolds often coincide, but if they don't, $\O_{min}(K)$ is the orientation double cover of $\O_{full}(K)$. 

The second key condition in the study of knot commensurability is the existence or not of {\it hidden symmetries} of a hyperbolic knot $K$. In other words, the existence or not of an isometry between finite degree covers of $S^3 \setminus K$ which is not the lift of an isometry of $S^3 \setminus K$. Equivalently, a hyperbolic knot complement admits hidden symmetries if and only if it covers some orbifold irregularly, or (for non-arithmetic knots) if and only if the minimal orientable orbifold in its commensurability class has a {\it rigid cusp}. (This means that the cusp cross section is a Euclidean turnover.) See \cite[Proposition 9.1]{NR1} for a proof of these equivalences. The latter characterisation shows that if $K_1$ and $K_2$ are commensurable, then $K_1$ admits hidden symmetries if and only if $K_2$ does. 


The commensurability classification of hyperbolic knots which admit no hidden symmetries was extensively studied in \cite{BBCW1}, where it was shown that such classes contain at most three knots. Further, there are strong constraints on the topology of a hyperbolic knot without hidden symmetries if its commensurability class contains more than one knot. For instance, it is fibred. 

To date, there are only three hyperbolic knots in $S^3$ which are known to admit hidden symmetries: the figure-eight and the two dodecahedral knots of Aitchison and Rubinstein described in \cite{AR}. Each is alternating, and the minimal orbifold in the commensurability class of the dodecahedral knot complements is a reflection orbifold (see \S 2). Using known restrictions on the trace field of a knot with hidden symmetries, Goodman, Heard, and Hodgson \cite{GHH} have verified that these are the only examples of hyperbolic knots with 15 or fewer crossings which admit hidden symmetries. 
This lends numerical  support to the following conjecture of W. Neumann and A. Reid. 

\begin{conjecture} {\rm (Neumann-Reid)} \label{conj:hidden} {\it The figure-eight knot and the two dodecahedral knots are the only hyperbolic knots in $S^3$ admitting hidden symmetries.} 
\end{conjecture}

\noindent In this article we investigate the Neumann-Reid conjecture in the context of non-arithmetic knots $K$ for which $\O_{full}(K)$ contains a reflection, a condition which implies that the knot admits hidden symmetries (Lemma \ref{l:prelim}). Our main results concern knot complements which cover hyperbolic reflection orbifolds, especially the complements of {\it AP knots}.  A knot $K$ is an AP knot if each closed essential surface in the exterior of $K$ contains an accidental parabolic (see \S \ref{s: ap knots} for definitions). Small knots (i.e.~knots whose exteriors contain no closed essential surfaces) are AP knots, but so are {\it toroidally alternating knots} \cite{Ad}, a large class which contains, for instance, all hyperbolic knots which are alternating, almost alternating, or Montesinos.  

\begin{theorem}\label{T:AP}
If the complement $S^3 \setminus K$ of a hyperbolic AP knot $K$ covers a reflection orbifold $\mathcal{O}$, then $\mathcal{O}$ is a one-cusped tetrahedral orbifold and $K$ is either the 
figure-eight  knot or one of the dodecahedral knots.
\end{theorem}

The cusp cross section of the orientable minimal orbifold in the commensurability class of a hyperbolic knot complement with hidden symmetries is expected to be $S^2(2,3,6)$. 

\begin{conjecture} {\rm (Rigid cusp conjecture)} \label{conj:cusp}  {\it The minimal orientable orbifold covered by a non-arithmetic knot complement with hidden symmetries has a rigid cusp of type $S^2(2,3,6)$.} 
\end{conjecture}

\noindent Here is a corollary of Theorem \ref{T:AP} that we prove in \S \ref{s: ap knots}. 

\begin{corollary}\label{c:achiral}
If the complement of an achiral, hyperbolic, AP knot $K$ covers an orientable orbifold with cusp cross section $S^2(2,3,6)$, then $K$ is the figure-eight or one of the dodecahedral knots.
\end{corollary} 

The proof of  Theorem \ref{T:AP} follows from the fact (Proposition \ref{P:small}) that a reflection orbifold covered by a hyperbolic AP knot complement cannot contain a closed essential 2-suborbifold (i.e.~an orbifold-incompressible 2-suborbifold which is not parallel to the cusp cross section) together with the following result: 

\begin{theorem}   \label{T:reflections} 
Suppose that $K$ is a hyperbolic knot.  If $S^3 \setminus K$ covers a reflection orbifold $\mathcal{O}$ which does not contain a closed, essential $2$-suborbifold, then $K$ is either the figure-eight knot and arithmetic or one of the dodecahedral knots and non-arithmetic. 
\end{theorem} 

\noindent Since the figure-eight knot is the only arithmetic knot and the dodecahedral knots are not small (\cite[Theorem 5]{Banks} and, independently,  \cite[Theorem 8]{computeclosed}), we obtain the following corollary. 

\begin{corollary} \label{not cover reflection}
No small, hyperbolic, non-arithmetic knot complement covers a reflection orbifold. 
\qed
\end{corollary} 

In \cite{BBCW1} we proved that the complements of two hyperbolic knots without hidden symmetries are commensurable if and only if they have a common finite cyclic covering, that is, they are {\it cyclically commensurable}. Hence two commensurable hyperbolic knot complements are cyclically commensurable or admit hidden symmetries.   Another consequence of Proposition \ref{P:small} is that the minimal orientable orbifold $\O_{min}(K)$ in the commensurability class of the complement of a non-arithmetic hyperbolic  AP knot $K$ which admits hidden symmetries is small (Corollary \ref{c:minimalsmall}). This allows us to extend results of N. Hoffman on small hyperbolic knots (\cite[Theorems 1.1 and 1.2]{Neilsmall}) to the wider class of AP knots.

\begin{theorem} \label{t: AP Hoffman}
Let $K$ be a non-arithmetic hyperbolic  AP knot which admits hidden symmetries. \\ 
\noindent $(1)$ $S^3 \setminus K$ is not cyclically commensurable with any other knot complement. \\
\noindent $(2)$ If $K$ admits two non-meridional, non-hyperbolic surgeries, then $S^3 \setminus K$ admits no non-trivial symmetry and $\O_{min}(K)$ has an $S^2(2,3,6)$ cusp cross section. Further, $\mathcal{O}_{min}(K)$ does not admit a reflection. \end{theorem}

\noindent This raises the following question, which was settled positively for hyperbolic 2-bridge knot complements in \cite{RW}. 

\begin{question} If a hyperbolic AP knot complement does not admit hidden symmetries, is it the unique knot complement in its cyclic commensurability class, 
hence in its commensurability class?
\end{question}

In our final result, we show that it is still possible to obtain strong restrictions on the combinatorics of the singular locus of $\O_{min}(K)$ if we replace the hypothesis that the minimal orbifold $\O_{full}(K)$ is a reflection orbifold by the weaker assumption that the orientable minimal orbifold  $\O_{min}(K)$ admits a {\it reflection} (i.e.~an orientation reversing symmetry with a 2-dimensional fixed point set).  

\begin{theorem} \label{thm:areflection} 
Let $K$ be a non-arithmetic hyperbolic knot and suppose that $\O_{min}(K)$ does not contain an essential $2$-suborbifold. Suppose further that $\O_{min}(K)$ admits a reflection. Then the $\O_{min}(K)$ is orbifold-homeomorphic to one of the following models: \\ 

\noindent $(a)$ a one-cusped tetrahedral orbifold; \\ \\ 

\noindent $(b)$ $Y333$ with $n = 2,3,4,5:$ 
 \vskip -1.35 in \hskip 1.5 in \includegraphics[width=2.05 in]{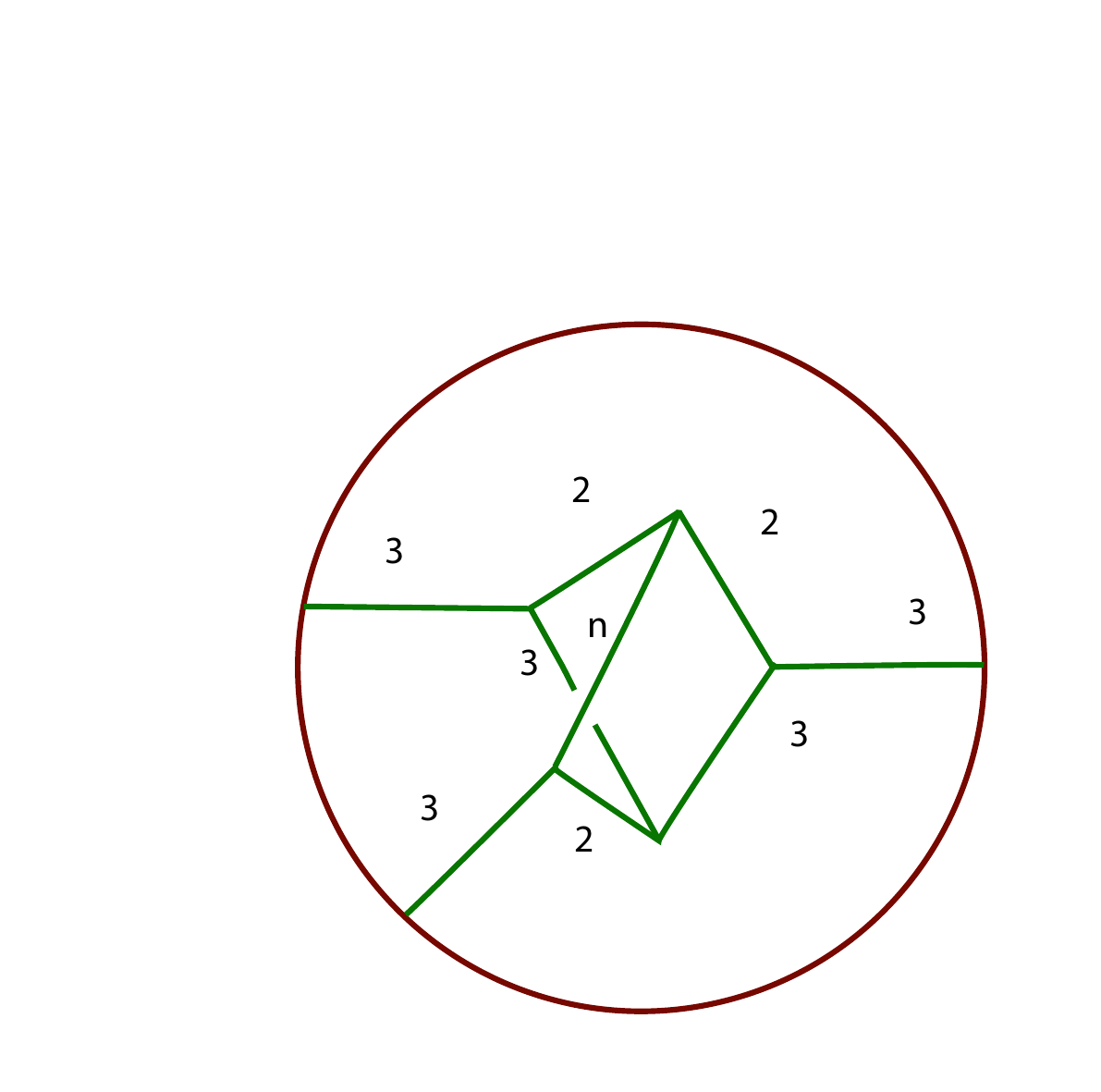} \\ \\ \\ 

\noindent $(c)$  $Y244$ with $n \geq 2:$
\vskip -2 in \hskip 1.1 in \includegraphics[width=3 in]{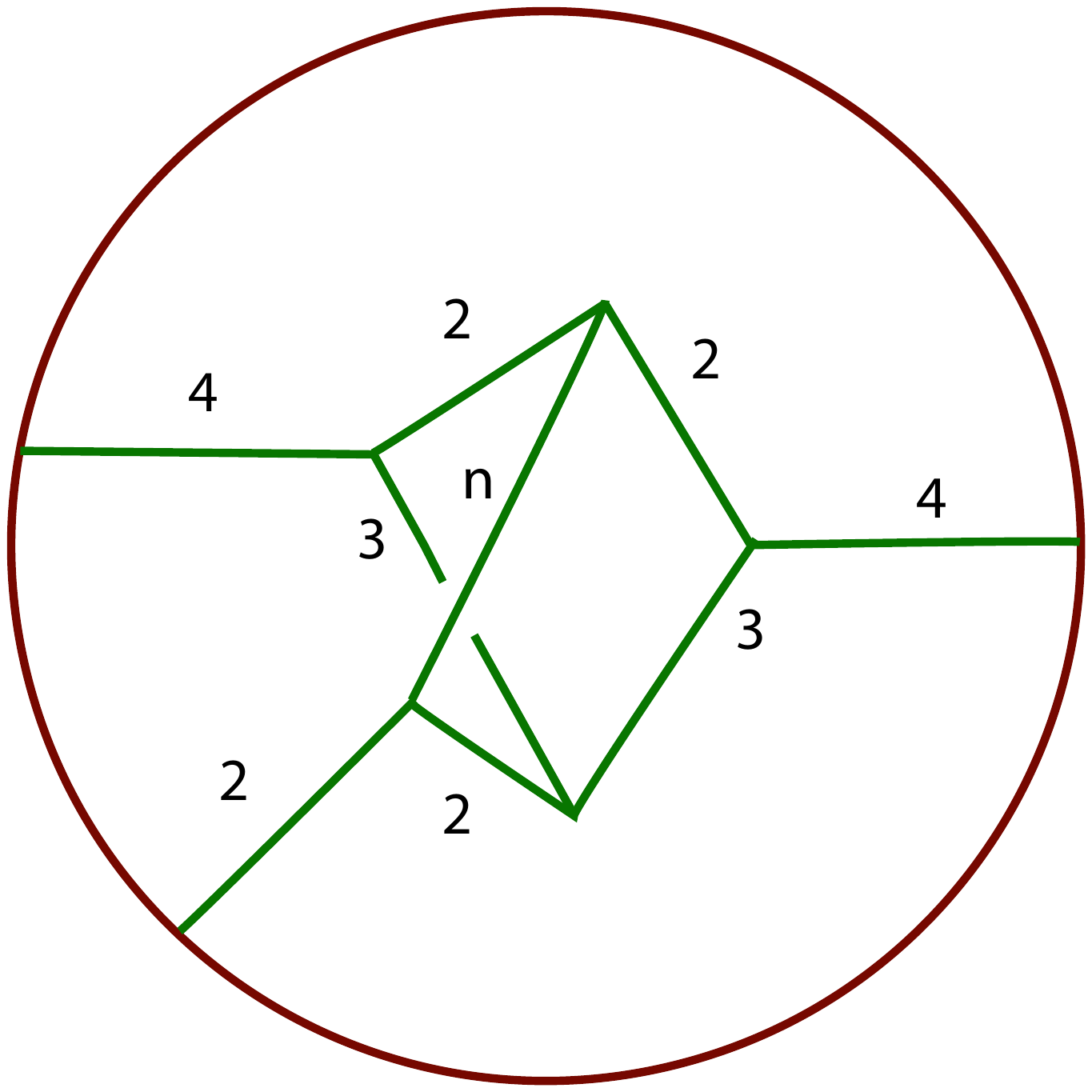} \\ \\ \\ 

\noindent $(d)$  $XO:$ 
\vskip -1.35 in \hskip 1.6 in  \includegraphics[width=2.1 in]{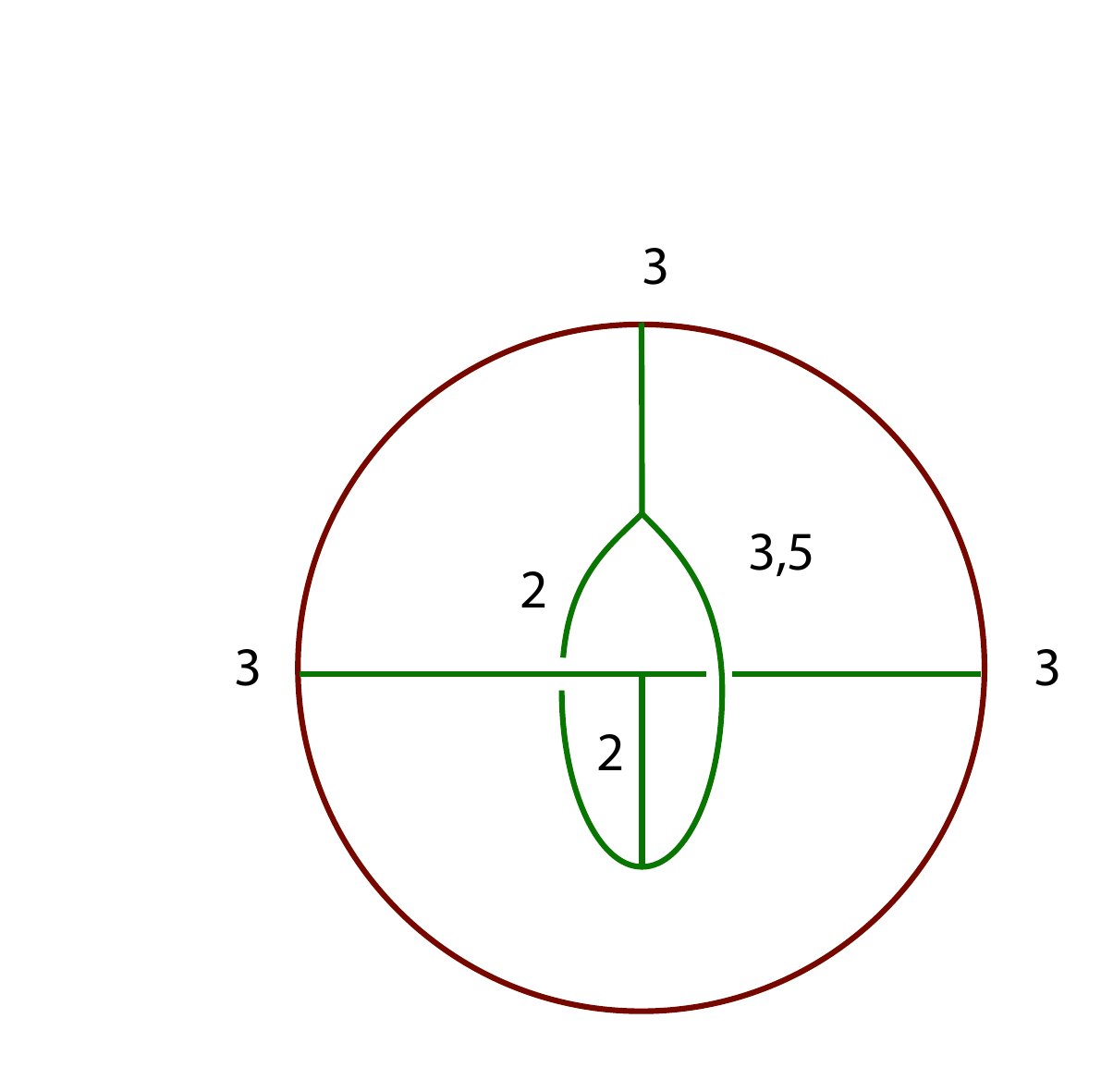}

\end{theorem} 

\begin{question}
Are any of the orbifolds listed above as type $(b)$, $(c)$, or $(d)$ in Theorem \ref{T:reflections} covered by a knot complement? 
\end{question}

\begin{conjecture} No. \end{conjecture} 

Here is how the paper is organized. In section \ref{general} we show that the presence of a reflection in the commensurability class of the complement of a hyperbolic knot $K$ implies that $K$ admits hidden symmetries (Lemma \ref{l:prelim}). Then we introduce the notion of a reflection orbifold and prove that for a non-arithmetic knot $K$, $\O_{full}(K)$ is a hyperbolic reflection orbifold as long its cusp cross section is a Euclidean reflection orbifold (Lemma \ref{l:reflection}). We also show that no hyperbolic knot complement covers a reflection orbifold regularly (Lemma \ref{l: quad type}(1)). In section \ref{s: ap knots} we consider hyperbolic AP knots whose complements cover reflection orbifolds (Theorem \ref{T:AP}) and show that if such a knot admits hidden symmetries, then it cannot be cyclically commensurable with another knot (Theorem \ref{t: AP Hoffman}).  In section 
\ref{s: kc&reflection orbifolds} we determine those knots whose complements cover small hyperbolic reflection orbifolds (Theorem \ref{T:reflections}).  
 Finally in section \ref{s: contain reflection} we study the combinatorics of the minimal orbifolds in the commensurability classes of knot complements where the orientable commensurator quotient admits a reflection, but whose full commensurator quotient is not a reflection orbifold (Theorem \ref{thm:areflection}). \\

\noindent {\bf Acknowledgements}.  The work presented here originated during a Research in Pairs program hosted by the Mathematisches Forschungsinstitut Oberwolfach. The authors would like to thank the institute for its hospitality.

\section{Knot complements and reflections} \label{general} 

\noindent In this section we prove some general results about knots whose complements have reflections in their commensurability classes.  For example, a hyperbolic amphichiral knot which admits hidden symmetries always has a reflection in its commensurability class.  This is because the orientable commensurator quotient has a cusp cross section with three cone points, and therefore must admit a reflection. We assume that the reader is familiar with the terminology and notation from \cite{BBCW1} and refer to \cite{Walshsurvey} for background information on orbifolds and how they relate to the knot commensurability problem. 

If $K \subset S^3$ is a hyperbolic knot, there is a discrete subgroup $\Gamma_K \leq PSL_2(\mathbb C)$, unique up to conjugation in $\hbox{Isom}(\mathbb H^3)$, such that $\pi_1(S^3 \setminus K) \cong \Gamma_K$. We use $M_K$ to denote the exterior of $K$. 

\begin{lemma} \label{l:prelim}
Suppose a non-arithmetic hyperbolic knot complement $S^3 \setminus K = \mathbb{H}^3/\Gamma_K$  is commensurable with an orbifold $\mathcal{O}$ which admits a reflection symmetry. Then, 

\noindent $(1)$ $S^3 \setminus K$ admits hidden symmetries. 

\noindent $(2)$ $S^3 \setminus K$ contains a (possibly immersed) totally geodesic surface. 

\end{lemma}

\begin{proof} 
First we prove (1). Let $r$ be the reflection symmetry in $\mathcal{O}$. Then $r$ is contained in the full commensurator $C(\Gamma_K) \leq \hbox{Isom}(\mathbb H^3)$ of $\pi_1(S^3 \setminus K) \cong \Gamma_K \leq PSL_2(\mathbb C)$. We claim that the full normalizer $N(\Gamma_K) \leq \hbox{Isom}(\mathbb H^3)$ of $\Gamma_K$ is a proper subgroup of $C(\Gamma_K)$. 
Otherwise, if the full commensurator is the full normalizer, then $r \Gamma_K r^{-1} = \Gamma_K$.  But no hyperbolic knot complement is normalized by a reflection.  This is because the normalizer quotient $N(\Gamma_K)/\Gamma_K$ is the group of isometries $Isom(S^3, K)$ and therefore $K$ can be isotoped so as to be invariant under a reflection in a $2$-sphere in $S^3$.  Since the non-trivial knot cannot be contained in the reflection sphere, it must meet the sphere in two points, which implies it is a connect sum.  Therefore, the knot complement admits a hidden reflection, and by \cite[Theorem 7.2]{BBCW1}, the knot admits hidden symmetries. This proves (1). 

For (2), suppose there is a reflection $\tau$ in $\pi_1^{orb}(\mathcal{O}_{full})$ though a plane $P$ in $\mathbb{H}^3$. Then $\Gamma_K \cap \tau \Gamma_K \tau^{-1}$ is a finite index subgroup of $\Gamma_K$which is normalized by $\tau$, as $\tau^2 =1$.  Thus $M_\tau = \mathbb{H}^3 / (\Gamma_K \cap \tau \Gamma_K \tau^{-1})$  admits a reflection and the fixed point set of this reflection is a totally geodesic surface.  As $M_\tau$ is a finite-sheeted cover of $S^3 \setminus K$, $S^3 \setminus K$ contains a totally geodesic surface.  This surface may be immersed and may have cusps. 
\end{proof} 

Here are some definitions needed for our next result. 

A  {\it convex hyperbolic polyhedron} is the intersection of a finite number of half-spaces in $\mathbb H^n$. 

A {\it hyperbolic reflection group} is a discrete subgroup of $\Isom(\mathbb H^n)$ generated by a finite number of reflections. 

Each hyperbolic reflection group $\Gamma \leq \Isom(\mathbb H^n)$ has a fundamental domain which is a convex hyperbolic polyhedron $P \subseteq \mathbb H^n$. Further, $\Gamma$ is generated by reflections in the faces of $P$. The quotient orbifold $\mathcal{O}_\Gamma = \mathbb{H}^n/\Gamma$ has underlying space $P$ and singular set $\partial P$. The faces of $P$ are reflector planes. See \cite[Theorem 2.1]{DOL} for a proof of these assertions. 

A {\em hyperbolic reflection orbifold} is an orbifold $\mathcal{O}$ associated to the quotient of $\mathbb{H}^n$ by a hyperbolic reflection group. 

{\it Euclidean reflection orbifolds} are defined similarly. We leave the details to the reader.

The orientation double cover $\TO$ of a one-cusped reflection $n$-orbifold $\mathcal{O}$ is obtained by doubling $\mathcal{O}$ along its reflector faces. Hence the interior of $\TO$ has underlying space $\mathbb R^n$ and singular set contained in a properly embedded hyperplane, which is the fixed point set of the reflection determined by the cover $\TO \to \mathcal{O}$. In the case of a $3$-dimensional reflection orbifold, the singular set of $\TO$ is a trivalent ``graph" with a finite number of vertices and a finite number of edges which are either lines or half-lines properly embedded in $\TO$, or compact intervals whose endpoints are vertices. (There are no circular edges since each loop contains a vertex.) 

\begin{lemma} \label{l:reflection} 
Let $K$ be a non-arithmetic hyperbolic knot. Suppose that the cusp cross section of the full commensurator quotient $\mathcal{O}_{full}(K)$ of $S^3 \setminus K$ is a Euclidean reflection 2-orbifold. Then $\mathcal{O}_{full}(K)$ is a hyperbolic reflection orbifold. 
\end{lemma} 

\begin{proof}
By Lemma \ref{l:prelim}(1), $S^3 \setminus K$ admits hidden symmetries and so by \cite[Corollary 4.11]{BBCW1}, the minimal element $\O_{min}(K)$ in the orientable commensurability class of $S^3\setminus K$ has underlying space a ball. Further, $\O_{min}(K)$ has a rigid cusp so its cusp cross section is a Euclidean 2-orbifold with 3 cone points. Thus it is of the form $S^2(2,3,6), S^2(2,4,4)$ or $S^2(3,3,3)$. 

The full commensurator quotient $\mathcal{O}_{full}(K)$ of $\mathbb H^3$ is the quotient of $\O_{min}(K)$ by an orientation-reversing involution, so the cusp cross section of $\mathcal{O}_{full}(K)$ is the quotient of the cusp cross section of $\mathcal{O}_{min}(K)$ by an orientation-reversing involution. Since we have assumed that this quotient is a Euclidean reflection orbifold, it must have underlying space a triangle with boundary made up of three reflector lines and three corner-reflectors. Thus the fundamental group of the cusp cross section of $\mathcal{O}_{full}(K)$ is a triangle group. It follows that the peripheral subgroup is a group generated by reflections so in particular, any meridional class of the knot is a product of reflections in the full commensurator. Since any conjugate of a product of reflections is a product of reflections, the knot group, which is normally generated by the meridian, is generated by products of reflections.  

Let $\Gamma_K, t_1\Gamma_K, t_2\Gamma_K, ...t_n\Gamma_K$ be the left cosets of $\Gamma_K$ in $\mathcal{O}_{full}(K)$. Since the one cusp of the knot complement covers the cusp of the commensurator quotient, the index of the covering restricted to the cusp is the same as the index of the cover. Thus  we may take our $t_i$ to be in the cusp group and therefore we may suppose that the $t_i$ are products of reflections. It follows that every element in the full commensurator quotient is a product of reflections in the full commensurator quotient, so the group is generated by reflections.
\end{proof} 

\begin{remark}  \label{236 implies reflection}
{\rm Suppose that $K$ is a non-arithmetic hyperbolic knot whose full commensurator contains an orientation reversing involution and whose orientable commensurator quotient has cusp cross section of the form 
$S^2(2,3,6)$. Since any involution of $S^2(2,3,6)$ fixes each of its cone points, the cusp cross section of $\mathcal{O}_{full}(K)$ is a Euclidean reflection 2-orbifold. Thus $\mathcal{O}_{full}(K)$ is a hyperbolic reflection orbifold by Lemma \ref{l:reflection}.} 
\end{remark}

\begin{lemma} \label{l: quad type}
A hyperbolic reflection orbifold, and its orientation double cover, cannot be regularly covered by a knot complement. 
\end{lemma}

\begin{proof}
Suppose that a knot complement $S^3 \setminus K$ regularly covers a hyperbolic reflection orbifold $\mathcal{O}$ and consider the orientation double cover $\TO$ of $\mathcal{O}$, which is also regularly covered by $S^3 \setminus K$. Since the interior of $\TO$ has underling space $\mathbb R^3$, its cusp cross section is $S^2(2,2,2,2)$. (The group of orientation-preserving isometries of a hyperbolic knot complement is cyclic or dihedral, by the positive solution of the Smith conjecture, so any orientable orbifold regularly covered by the knot complement has cusp cross section a torus or $S^2(2,2,2,2)$.) Thus $\TO$ is the quotient of $S^3 \setminus K$ by a dihedral group generated by a strong inversion of $K$ and a cyclic symmetry $\sigma$ of order $n \geq 1$ whose axes are disjoint from $K$. 

We noted above that the singular set $\Sigma(\TO)$ of $\TO$ is contained in the reflection $2$-plane $P$ of the cover $\TO \to \mathcal{O}$. Since the cusp cross section of $\TO$ is $S^2(2,2,2,2)$, the image in $\TO$ of the axis of the strong inversion is a pair of disjoint properly embedded real lines $L_1$ and $L_2$.    This cannot be all of $\Sigma(\TO) \subset P$, as it does not form the edges of a polyhedron with one ideal vertex.  Thus $\sigma$ has order $n > 1$. Since the quotient $(S^3 \setminus K)/\langle \sigma\rangle$ has singular set consisting of one or two circles, while, as we noted above, each loop in $\Sigma(\TO)$ has a vertex, the axis of the strong inversion meets each of the axes of the cyclic symmetry in two points. It follows that $\Sigma(\TO)$ is connected and is the union of $L_1, L_2$ and one arc for each of the circles in the singular set of $(S^3 \setminus K)/\langle \sigma\rangle$. The reader will verify that the hyperbolicity of $S^3 \setminus K$ forces $\Sigma(\TO)$ to be the union of $L_1, L_2$ and two arcs, each running from $L_1$ to $L_2$. By construction, the sides of the quadrilateral labeled contained in $L_1$ and $L_2$ are disjoint and labeled $2$. But then there is an essential $D^2(2,2)$ properly embedded in $\TO^{tr}$, the truncation of $\TO$ along an $S^2(2,2,2,2)$ cusp cross section, contrary to the hyperbolicity of $\TO$. This completes the proof. 
\end{proof}

\section{AP knots} \label{s: ap knots}

\noindent 
The main result of this section is given by Proposition \ref{P:small}, which allows us to reduce the proof of Theorem \ref{T:AP} to that of Theorem \ref{T:reflections}. We begin with some definitions.

\begin{definition} Let $M$ be a compact $3$-manifold with incompressible boundary. \\ 
$(1)$ An {\bf accidental parabolic} of an essential surface $S$ in $M$ is an essential loop on $S$ which is homotopic in $M$ to a peripheral curve of $M$. \\ 
$(2)$ An {\bf AP knot} is a knot $K$ in $S^3$ such that any closed essential surface in the exterior of $K$ contains an accidental parabolic.  

\end{definition}


\noindent The class of AP knots includes {\it small knots} (i.e.~knots whose exteriors contain no closed essential surfaces) and toroidally alternating knots, a class which contains all alternating knots, all almost alternating knots, and all Montesinos knots. See  \cite{Ad}.

Let $K$ be an AP knot. If $S$ is a closed essential surface in the exterior $M_K$ of $K$ and $N$ is the component of $M_K$ cut open along $S$ which contains $\partial M_K$, then the annulus theorem implies that there is an essential annulus  properly embedded in $N$ with boundary the union of an essential simple closed curve on $S$ and an essential simple closed curve on $\partial M_K$.  In the case that $K$ is a toroidally alternating knot, we can take the essential simple closed curve on $\partial M_K$ to be a meridional curve of $K$ (\cite{Ad}).

A 2-suborbifold $\mathcal{F}$ of a 3-orbifold $\mathcal{O}$ is {\it orbifold-incompressible} if there is no orbifold-essential curve on $\mathcal{F}$ that bounds an orbi-disc in $\mathcal{O}$, 
$\mathcal{F}$ is not a spherical orbifold which bounds an orbi-ball, and $\mathcal{F}$ is finitely covered by a surface. It is {\it essential} if it is orbifold-incompressible and not boundary-parallel. 

An obifold is {\em small} if it is irreducible and contains no closed, essential  2-suborbifold.  Note that a small manifold cannot cover an orbifold which contains an essential closed $2$-suborbifold.  

\begin{proposition}\label{P:small} If the complement of an AP knot $K$ finitely covers an orientable orbifold $\mathcal{O}$ with rigid cusp, then $\mathcal{O}$ does not contain any closed essential orientable $2$-suborbifold, and thus is small.
\end{proposition}

\begin{proof} 
Let $\mathcal{O}^{tr}$ be the result of truncating $\mathcal{O}$ along a Euclidean turnover cross section of its cusp, and let $M_K$ be the inverse image of $\mathcal{O}^{tr}$ in $S^3 \setminus K$. Then $M_K$ is the exterior of $K$. Let $\pi: M_K \to \mathcal{O}^{tr}$ be the covering map. In order to obtain a contradiction, we assume that $\hbox{int}(\mathcal{O}^{tr})$ contains a closed, essential, orientable, connected 2-suborbifold $\mathcal{F}$. Since $\mathcal{O}^{tr}$ is finitely covered by a hyperbolic knot exterior, $\mathcal{F}$ has negative Euler characteristic and separates $\mathcal{O}^{tr}$. The 2-suborbifold $\mathcal{F}$ splits $\mathcal{O}^{tr}$ into two compact 3-suborbifolds $\mathcal{O}_1$ and $\mathcal{O}_2$, where $\partial \mathcal{O}_1 = \mathcal{F} \cup \partial \mathcal{O}^{tr}$. The preimage $F=\pi^{-1}(\mathcal{F})$ is a closed orientable (possibly disconnected) essential surface in $\partial M_1$ where $M_1 = \pi^{-1}(\mathcal{O}_1)$. Note that as $\pi^{-1}(\partial \mathcal{O}^{tr}) = \partial M_K$ is connected, $M_1$ is a compact, connected submanifold of $M_K$. In particular, it is the component of $M_K$ split open along $F$ with boundary $\partial M_1 = F \cup \partial M_K$.

A closed Euclidean $2$-suborbifold is called {\it canonical} if it can be isotoped off any essential closed Euclidean 2-suborbifold. For instance, Euclidean turnovers are always canonical (\cite[Remark, page 47]{BMP}). By the JSJ theory of $3$-orbifolds, any maximal collection of disjoint, canonical, essential, closed Euclidean $2$-suborbifolds in a compact, irreducible $3$-orbifold $\mathcal{O}$ is finite (\cite[Theorem 3.11]{BMP}) and unique up to isotopy (\cite[Theorem 3.15]{BMP}). It follows that each isotopy class of Euclidean turnovers in $\mathcal{O}$ is contained in this collection. If $\mathcal{O}$ contains no bad $2$-suborbifolds, the solution of the geometrisation conjecture implies that either $\mathcal{O}$ is a closed $Sol$ orbifold or the collection splits $\mathcal{O}$ into geometric pieces. See \cite[\S 3.7]{BMP}.

Doubling $\mathcal{O}_1$ along its boundary produces a closed, connected, irreducible orbifold $D(\mathcal{O}_1)$ which is finitely covered by the double $D(M_1)$ of $M_1$. Thus $D(\mathcal{O}_1)$ is irreducible, contains no bad $2$-suborbifolds, and contains an essential suborbifold of negative Euler characteristic (i.e. any component of $\mathcal{F}$). The latter fact implies that $D(\mathcal{O}_1)$ is not a $Sol$ orbifold. It follows that a maximal collection of disjoint, canonical, essential, closed Euclidean $2$-suborbifolds of $D(\mathcal{O}_1)$ splits it into geometric pieces. Since boundary components of a Seifert piece are tori or copies of $S^2(2,2,2,2)$, the pieces which are incident to $\partial \mathcal{O}_1$ must be hyperbolic. 

The geometric splitting of $D(\mathcal{O}_1)$ lifts to a geometric splitting of $D(M_1)$ and, from the previous paragraph, the geometric pieces containing $\partial M_K$ are hyperbolic. Thus $\partial M_K$ is a JSJ torus in $D(M_1)$ and so any incompressible torus in $D(M_1)$ can be isotoped into its complement.  On the other hand, since $K$ is an AP knot, there is a properly embedded essential annulus $A$ in $M_1$ running from $F = \partial M_1 \setminus \partial M_K$, to $\partial M_K$. The double $T = D(A)$ of this annulus is a torus in $D(M_1)$ which meets the torus $\partial M_K$ along an essential simple closed curve, so that $T$ is non-separating in $D(M_1)$. It follows that $T$ is incompressible, as otherwise the irreducible manifold $D(M_1)$ would contain a non-separating $2$-sphere. The double of the co-core of $A$ is a simple closed curve on $T$ which meets $\partial M_K$ transversely in a single point, and hence is homologically dual to $\partial M_K$. But then $T$ cannot be isotoped off of $\partial M_K$, which contradicts our observations above. Thus $\mathcal{O}$ does not contain any closed essential orientable $2$-suborbifold, which completes the proof. 
\end{proof}
 
Here is an immediate corollary of Proposition \ref{P:small} and the fact that the minimal orientable orbifold in the commensurability class of a  non-arithmetic hyperbolic knot complement with hidden symmetries has a rigid cusp (\cite[Proposition 9.1]{NR1}). 
 
 \begin{corollary}\label{c:minimalsmall} If the complement of a non-arithmetic hyperbolic  AP knot $K$ admits hidden symmetries, then the minimal orientable orbifold $\O_{min}(K)$ in its commensurability class is small. 
 \qed 
 \end{corollary}

\begin{proof}[Proof of Theorem \ref{T:AP} modulo Theorem \ref{T:reflections}] 
Let $K$ be a hyperbolic AP knot and suppose that $S^3 \setminus K$ finitely covers a reflection orbifold $\mathcal{O}$. We claim that the orientation double cover $\TO$ of $\mathcal{O}$, which is also covered by $S^3 \setminus K$, has a rigid cusp. Suppose otherwise and note that as the cusp cross section of $\mathcal{O}$ has underlying space a disk, the cusp cross section of $\TO$ is necessarily $S^2(2,2,2,2)$, and so the cover $S^3 \setminus K \to \widetilde{\mathcal{O}}$ is regular (\cite[\S 6.2]{Re1}), contrary to Lemma \ref{l: quad type}. Thus $\widetilde{O}$ has a rigid cusp, and by Proposition \ref{P:small} it cannot contain a 
closed essential orientable 2-suborbifold. Then $\mathcal{O}$ cannot contain a closed essential 2-suborbifold.
It follows that $S^3 \setminus K$ must cover a small reflection orbifold. Theorem \ref{T:AP} now follows from Theorem \ref{T:reflections}, whose proof is contained in  \S \ref{s: kc&reflection orbifolds}.
\end{proof}

The proof of  Corollary \ref{c:achiral} follows from Theorem \ref{T:AP} and Remark \ref{236 implies reflection}. 

\begin{proof}[Proof of Corollary \ref{c:achiral}]
Since the complement of $K$ covers an orientable orbifold with cusp cross section $S^2(2,3,6)$, its minimal orientable orbifold $\O_{min}(K)$ has cusp cross section $S^2(2,3,6)$. 
Moreover $K$ admits an orientation reversing symmetry, hence its full commensurator is strictly bigger than its orientable commensurator and thus 
the full commensurator quotient $\mathcal{O}_{full}(K)$ is the quotient of the orientable 
minimal orbifold $\O_{min}(K)$ by an orientation reversing involution.
Since any involution of the turnover $S^2(2,3,6)$ fixes each of its cone points, the cusp cross section of  $\mathcal{O}_{full}(K)$ is a 
Euclidean reflection 2-orbifold. Thus $\mathcal{O}_{full}(K)$ is a hyperbolic reflection orbifold by Lemma \ref{l:reflection} and the result follows from Theorem \ref{T:AP}.
\end{proof} 

\begin{proof}[Proof of Theorem \ref{t: AP Hoffman}] 
Hoffman proved part (1) of Theorem \ref{t: AP Hoffman} for small hyperbolic knots (\cite[Theorem 1.1]{Neilsmall}), though the smallness of $S^3 \setminus K$ is only used to deduce that the lattice $\Gamma_{\O_{min}(K)} \leq PSL_2(\mathbb C)$, corresponding to the fundamental group of $\O_{min}(K)$, has integral traces. Since the smallness of $\O_{min}(K)$ suffices to assure this last property, Corollary \ref{c:minimalsmall} allows us to extend Hoffman's result to the case of AP knots. This proves (1). Similarly the claims in the first sentence of Theorem \ref{t: AP Hoffman}(2) follows from the proof of \cite[Theorem 1.2]{Neilsmall}. By remark \ref{236 implies reflection} and Theorem \ref{T:AP}, if $\O_{min}(K)$ admits a reflection, it covers a reflection orbifold and so is one of the dodecahedral knots.  But these both admit an orientation-preserving non-trivial symmetry. 
\end{proof} 

An {\it APM knot} is an AP knot $K$ such that any closed essential surface in $S^3 \setminus K$ carries an essential curve which is homotopic to a meridian  of $K$. Examples of APM knots include all toroidally alternating knots. The next proposition puts constraints on the minimal orientable orbifold in the commensurability classes of certain APM knot complements.

\begin{proposition}  \label{p:apm small}
If two distinct APM knot complements cover an orbifold $\mathcal{O}$ with a flexible cusp, then $\mathcal{O}$ is small. 
\end{proposition}  

\begin{proof} 
Let $K_1, K_2$ be the two APM knots. By the proof of Lemma 4.3 of \cite{BBCW1} (see also \cite[Remark 4.4]{BBCW1}), their complements cover an orbifold with a torus cusp which, without loss of generality, we take to be $\mathcal{O}$. By \cite[Corollary 4.11]{BBCW1}, $|\mathcal{O}|$ is the complement of a knot in a lens space, and by \cite[Theorem 1.1]{GAW}, the knot complements $S^3 \setminus K_1, S^3 \setminus K_2$ cover $\mathcal{O}$ cyclically. Further, the images in $\O^{tr}$ of the two meridians $\mu_{K_1}, \mu_{K_2}$ of the knots represent primitive classes $\bar{\mu}_{K_1}, \bar{\mu}_{K_2}$ of intersection number 1 on the boundary torus of $\mathcal{O}^{tr}$ (\cite[Lemma 4.8]{BBCW1}). 

Consider a closed essential $2$-suborbifold $\mathcal{S}$ contained in the interior of $\mathcal{O}$. Then $|\mathcal{S}|$ is a closed submanifold of the orientable $3$-manifold $|\mathcal{O}|$, and as the latter is contained in a lens space, $|\mathcal{S}|$ is separating. Let $\mathcal{N}$ be the component of  $\mathcal{O}^{tr}$ cut open along $\mathcal{S}$ which contains $\partial \O^{tr}$. Next let $S_j$ be the inverse image of $\mathcal{S}$ in $M_{K_j}$, a closed, essential, separating surface. Let $N_j$ be the component of $M_{K_j}$ cut open along $S_j$ which contains $\partial M_{K_j}$. Then $(N_j, \partial N_j) \to (N, \partial N)$ is a finite cyclic cover for both $j$. 

Fix a base-point in $\partial \O^{tr}$ and lifts of it to $\partial M_{K_1}$ and $\partial M_{K_2}$. By hypothesis there is an essential annulus in $N_j$ which intersects $\partial M_{K_j}$ in a meridional curve and $S_j$ in an essential simple closed curve. In particular $\bar{\mu}_{K_2} \in \pi_1^{orb}(\mathcal{N})$ is represented by a loop in $\mathcal{N}$ of the form $\alpha \* \gamma \* \alpha^{-1}$ where $\alpha$ is a path in $\mathcal{N}$ connecting $\partial \O^{tr}$ to $\mathcal{S}$ and $\gamma$ is a loop in $\mathcal{S}$. If $n$ is the index of $\pi_1(S^3 \setminus K_1)$ in $\pi_1^{orb}(\mathcal{O})$, then $\bar{\mu}_{K_2}^n$ lifts to a class $\mu_2$ in $\pi_1(N_1)$ represented by a loop in $\partial M_{K_1}$ which is rationally independent of the class of $\mu_{K_1}$ in $H_1(\partial M_{K_2})$. Further, $\alpha \* \gamma^n \* \alpha^{-1}$ lifts to a loop of the form $\widetilde{\alpha} \* \beta \* \widetilde{\alpha}^{-1}$ where $\widetilde{\alpha}$ is a path in $N_1$ connecting $\partial M_{K_1}$ to $S_1$ and $\beta$ is a loop in $S_1$. Hence there is a singular annulus in $N_1$ which represents a homotopy between the loop $\mu_2$ in $\partial M_{K_1}$ and the loop $\beta$ in $S_1$. On the other hand, by hypothesis there is an embedded annulus in $N_1$ representing a homotopy between $\mu_{K_1}$ and another loop in $S_1$.  Hence there is a component of the characteristic Seifert pair of $(N_1, \partial N_1)$ homeomorphic to $S^1 \times S^1 \times I$ such that $S^1 \times S^1 \times \{0\}$ corresponds to $\partial M_{K_1}$ and $S^1 \times S^1 \times \{1\}$ corresponds to a subsurface of $S_1$. But this impossible as it implies that $S_1$ has a component homeomorphic to a torus, contrary to the hyperbolicity of $M_{K_1}$.  Thus $\mathcal{O}$ does not contain a closed essential $2$-suborbifold.
\end{proof}

\begin{corollary}\label{c:minimalorbi} If a commensurability class contains two distinct hyperbolic APM knot complements, then the minimal orientable orbifold in this commensurability 
class is small. 
\end{corollary}

\begin{proof} If the minimal orientable orbifold has a rigid cusp, the result follows from Proposition \ref{P:small}. Otherwise it has a flexible cusp, and we apply Proposition \ref{p:apm small}.
\end{proof}




\section{Knot complements and reflection orbifolds} \label{s: kc&reflection orbifolds}

The goal of this section is to prove Theorem \ref{T:reflections}. We begin with a 
characterisation of the combinatorial type of small one-cusped reflection orbifolds.  

We say that a orbifold $\mathcal{O}$ is a \emph{one-cusped tetrahedral orbifold} if the  orbifold fundamental group $\pi_1^{\orb}(\mathcal{O})$ is generated by reflections in the faces of a tetrahedron with one ideal point. 

We say that $\mathcal{O}$ is a \emph{one-cusped orbifold of quadrilateral type} if $\pi_1^{\orb}(\mathcal{O})$ is generated by reflections in the faces of a cone over a quadrilateral where the cone point is ideal. 



\begin{lemma} \label{l:barrier} {\rm (The barrier lemma)}  Let $\mathcal{O}$ be a  3-orbifold whose interior has underlying space an open $3$-ball and which admits a complete finite volume hyperbolic structure. Let $\mathcal{O}^{tr}$ be the orbifold with boundary obtained by truncating the cusps of $\mathcal{O}$. Suppose that the singular set $\Sigma(\mathcal{O}^{tr})$ contains two disjoint $1$-cycles rel boundary which are separated by a $2$-sphere in $|\mathcal{O}|$. Then $\mathcal{O}$ contains a orbifold-incompressible 2-suborbifold. 
\end{lemma} 

\begin{proof} Call the two 1-cycles $a$ and $b$. By hypothesis there is  a 2-suborbifold $\mathcal{S}$ of $\mathcal{O}$ such that 
\begin{enumerate} 
\item $|\mathcal{S}| \cong S^2$;
\item $|\mathcal{S}|$ separates $a$ and $b$; 
\item $\mathcal{S}$ meets $\Sigma(\mathcal{O}^{tr})$ transversely and in the interior of its edges;
\item $\mathcal{S}$ has the minimal number of cone points among all 2-suborbifolds satisfying the first three conditions.  
\end{enumerate} 
We will prove that $\mathcal{S}$ is is the desired $2$-suborbifold. Recall that  a 2-suborbifold $\mathcal{S} \subset \mathcal(O)$ is {\it orbifold-incompressible} if there is no orbifold-essential curve on $\mathcal{S}$ that bounds an orbi-disc in $\mathcal{O}$, 
$\mathcal{S}$ is not a spherical orbifold which bounds an orbi-ball, and $\mathcal{S}$ is finitely covered by a surface. 

Since $\mathcal{O}$ is finitely covered by a hyperbolic 3-manifold, $\mathcal{S}$ is finitely covered by a surface.  The orbifold $\mathcal{S}$ cuts $\mathcal{O}^{tr}$ into two pieces, one of which contains $\partial \mathcal{O}^{tr}$. This piece cannot be an orbi-ball.  The piece not containing $\partial \mathcal{O}^{tr}$ cannot be an orbi-ball either, since its singular set contains a cycle which does not meet $\mathcal{S}$ and the singular set of an orbi-ball is a single arc or a tripod.  It remains to show that $\mathcal{S}$ does not admit a compressing orbi-disk.

Suppose that $\mathcal{S}$ admits a compressing orbi-disk $\mathcal{D}$. Since $|\mathcal{O}|$ is a ball, $|\mathcal{D}|$ intersects $a$ and $b$ zero times algebraically. Hence it is disjoint from $a$ and $b$ as it has at most one cone point. Since $\mathcal{S}$ has underlying space a sphere, $\partial \mathcal{D}$ bounds two discs on $\mathcal{S}$, each of which has at least two cone points as $\partial \mathcal{D}$ is essential on $\mathcal{S}$. One of the two $2$-suborbifolds constructed from $\mathcal{D}$ and these two discs on $\mathcal{S}$ satisfies conditions (1), (2) and (3) and has fewer cone points, contrary to (4). Thus $\mathcal{S}$ is orbifold-incompressible. 
\end{proof} 

\begin{lemma} \label{l:charsmall} 
Let $\mathcal{O}$ be a one-cusped hyperbolic reflection orbifold which does not contain a closed orbifold-incompressible $2$-suborbifold. Then $\mathcal{O}$ is either a one-cusped tetrahedral orbifold or a one-cusped orbifold of quadrilateral type. 
\end{lemma} 

\begin{proof}  An alternative proof of this lemma can be deduced from Thurston's notes \cite{Th}, particularly from the work contained in chapter 13.  

By hypothesis, $|\Sigma(\mathcal{O})| \cong \mathbb R^2$ and as the orientation double cover $\widetilde{\mathcal{O} } \rightarrow \mathcal{O}$ is obtained by doubling $\mathcal{O}$ along $\Sigma(\mathcal{O})$, 
\begin{itemize}
\item the covering group of the cover $\widetilde{\mathcal{O} }\rightarrow \mathcal{O}$ is generated by a reflection $r$;
\item $\Sigma(\widetilde{\mathcal{O}})$ is contained in the reflection plane of $r$; 
\item $|\widetilde{\mathcal{O}}| \cong \mathbb R^3$. 
\end{itemize}
Further observe that a cross section of the cusp of $\mathcal{O}$ is a Euclidean reflection $2$-orbifold and as such is generated by reflections in the faces of either a triangle or a quadrilateral. Thus the end of $\widetilde{\mathcal{O}}$ has cross section a Euclidean $2$-orbifold with underlying space a $2$-sphere and either three or four cone points. Let $\widetilde{\mathcal{O}}^{tr}$ denote the truncation of $\widetilde{\mathcal{O}}$. The boundary of $\widetilde{\mathcal{O}}^{tr}$ is the double of the cusp cross section of $\widetilde{\mathcal{O}}$ and so is either $S^2(2,3,6), S^2(2,4,4), S^2(3,3,3)$ or $S^2(2,2,2,2)$. The intersection of the reflection plane of $r$ with $\widetilde{\mathcal{O}}^{tr}$ is a disc $P$. 

Set $\Sigma = P \cap \Sigma(\widetilde{\mathcal{O}}) = \Sigma(\widetilde{\mathcal{O}}^{tr})$, since $\mathcal{O}$ is a reflection orbifold.  By construction $\Sigma \cap \partial P$ consists of the three or four cone points which we call the {\it boundary vertices of $\Sigma$}. Further, $\Sigma$ admits the structure of a graph whose vertices have valency $3$ when contained in $\hbox{int}(P)$ (the {\it interior vertices}) and valency $1$ otherwise. If $\Sigma$ has a component contained in $\hbox{int}(P)$, it is separated from the other components of $\Sigma$ by a circle embedded in $\hbox{int}(P)$. But this is impossible as otherwise there would be a reducing $2$-sphere in $\widetilde{\mathcal{O}}$. Hence if $\Sigma$ is not connected, there is a properly embedded disc $D$ in $\widetilde{\mathcal{O}}^{tr}$ whose intersection with $P$ is a properly embedded arc which separates $P$ into two pieces, each containing components of $\Sigma$. Now $\partial D$ splits the $2$-sphere $|\partial \widetilde{\mathcal{O}}^{tr}|$ into two discs, each containing at least one point. In fact each contains two cone points as otherwise $\widetilde{\mathcal{O}}$ would contain a bad $2$-suborbifold. But then $\partial D$ is essential in $\partial \widetilde{\mathcal{O}}^{tr}$, which is impossible. Thus $\Sigma$ is connected. 

Suppose that $\Sigma$ contains an arc $a$ connecting two points on $\partial P$ and an absolute cycle $b$ which is disjoint from $a$. It is easy to see that there is a $2$-sphere in $|\widetilde{\mathcal{O}}|$ which separates $a$ and $b$ and therefore $\widetilde{\mathcal{O}}$ contains an orbifold-incompressible $2$-suborbifold by Lemma \ref{l:barrier}, contrary to our hypotheses. Thus each (absolute) cycle in $\Sigma$ contains at least one vertex of any arc in $\Sigma$ connecting two points of $\partial P$.

Next we observe that as all interior vertices of $\Sigma$ have valency $3$, there is a unique outermost embedded arc in $\Sigma$ connecting any two of its boundary vertices which are adjacent on $\partial P \cong S^1$. Note that distinct outermost arcs which share an endpoint share the edge connecting that endpoint to an interior vertex. Number these outermost arcs $a_1, \ldots, a_n$ where $n = 3$ or $n = 4$ and $a_i$ shares exactly one endpoint with $a_{i+1}$, where the indices are taken $\mod n$.   
 
We claim each $a_i$ has three edges. Indeed, it is easy to construct a boundary compressing orbi-disc in $\widetilde{\mathcal{O}}^{tr}$ if some $a_i$ has two edges. Suppose then that there are more than three on some $a_i$. Then there are at least three interior vertices of $a_i$, say $v_1$, $v_2$ and $v_3$, where $v_1$ and $v_3$ are adjacent to $\Sigma \cap \partial P$ and $v_2$ is not. Consider the edge $e_1$ of $\Sigma$ incident to $v_2$ but not contained in $a_i$. Since the interior vertices of $\Sigma$ have valency $3$, there is an oriented edge-path $e_1, e_2, \ldots, e_m$ in $\Sigma$ such that 
\begin{itemize}
\item $e_k$ and $e_{k+1}$ are distinct edges for each $k$ ;
\item it connects $a_i$ to some $a_j$ (it is possible that $i = j$); 
\item $e_2, \ldots, e_{m-1}$ are disjoint from $a_1 \cup \ldots \cup a_n$. 
\end{itemize}
It is easy to see that there is an $l \ne i, j$ such that $a_l$ is disjoint from an absolute cycle contained in $a_i \cup e_1 \cup \ldots \cup e_m \cup a_j$, contrary to what we deduced above. Thus each $a_i$ has exactly three edges and the theorem follows quickly from this observation. 
\end{proof} 

We can now prove Theorem \ref{T:reflections}. 

\begin{proof}[Proof of Theorem \ref{T:reflections}]
First we show that $\mathcal{O}$ is a one-cusped tetrahedral orbifold. By Lemma \ref{l:charsmall} we need only show that $\mathcal{O}$ is not a one-cusped orbifold of quadrilateral type. 
Suppose otherwise. Then the cusp cross section of the orientation double cover $\widetilde{\mathcal{O}} \rightarrow \mathcal{O}$ is $S^2(2,2,2,2)$. The cover $S^3 \setminus K \to \mathcal{O}$ factors through $\widetilde{\mathcal{O}}$ so by \cite[\S 6.2]{Re1}, the induced cover $S^3 \setminus K \to \widetilde{\mathcal{O}}$ is regular and dihedral. This contradicts Lemma \ref{l: quad type}. Thus $\mathcal{O}$ is a one-cusped tetrahedral orbifold. 

The cusp cross section of $\widetilde{\mathcal{O}}$ is either $S^2(2,3,6)$,  $S^2(3,3,3)$ or $S^2(2,4,4)$. Our strategy is to determine the indices of the singular locus of $\widetilde{\mathcal{O}}$. Here is an immediate constraint they satisfy. 
\begin{eqnarray}
\hbox{\hspace{-2.5cm} The link of each interior vertex is a spherical 2-orbifold.}
\end{eqnarray}
Hoffman \cite[Proposition 4.1]{Neilsmall} gives further restrictions on $H_1(\widetilde{\mathcal{O}})$ which hold for any orbifold covered by a knot complement.  Namely: 
\begin{eqnarray}
\hbox{\hspace{-.7cm}  $H_1(\widetilde{\mathcal{O}})$ is a quotient  of $\Z/2 \Z$ if $\widetilde{\mathcal{O}}$ has cusp cross section $S^2(2,3,6)$.}
\end{eqnarray}
We consider each possible cusp cross section in turn. Call an edge of $\Sigma(\widetilde{\mathcal{O}})$ {\it peripheral} if it is one of the non-compact edges which is properly embedded in the cusp of $\widetilde{\mathcal{O}}$. 

First suppose that the cusp cross section is $S^2(2,3,6)$. Since $\O$ is tetrahedral,  the indices on the edges of the interior triangle of the singular set $\Sigma(\widetilde{\mathcal{O}})$ determine the orbifold.  By restriction (1), the two edges of the triangle meeting the peripheral edge of $\Sigma(\widetilde{\mathcal{O}})$ labeled $6$ are  labeled 2. Similarly the third edge is labeled $2,3,4,$ or $5$. The third edge cannot be labeled $2$ or $4$ as otherwise $H_1(\widetilde{\mathcal{O}})$ is $\Z/2 \Z \times \Z/2 \Z$, violating restriction (2) above. If it is labeled $3$, $\O$ is an arithmetic orbifold and is covered by
 the figure-eight knot complement, and hence $K$ is the figure-eight knot (combine \cite[\S 2]{NR2} and \cite{Re1}).  If the third edge is labeled 5, $\TO$ is the minimal orbifold in the commensurability class of the dodecahedral knot complements \cite[\S 9]{NR1}. By the main result of \cite{Neildodec}, $K$ is one of the dodecahedral knots. In both cases $\widetilde{\mathcal{O}}$ is a minimal element of this orientable commensurability class. 
 
 If the cusp-cross section is $S^2(3,3,3)$, restriction (1) implies that two of the edges of the interior triangle are labeled 2. But then $\widetilde{\mathcal{O}}$ admits a reflection with quotient the orientation double cover of a one-cusped tetrahedral orbifold and this double cover has cusp cross section $S^2(2,3,6)$. Thus $K$ is the figure-eight knot. 
 
Finally suppose that the cusp cross section is $S^2(2,4,4)$. We will show that this assumption leads to a contradiction. 

Without loss of generality we can suppose that $\mathcal{O}$ does not cover another reflection orbifold non-trivially. By restriction (1), at least one of the two edges of the triangle meeting a peripheral edge labeled $4$ is labeled 2. The two edges of the triangle incident to the peripheral edge labeled $2$ cannot both have the same label as otherwise $\widetilde{\mathcal{O}}$ would admit a reflection with quotient the orientation double cover of a one-cusped tetrahedral orbifold. In particular one of these edges  has label $3$ or more. By restriction (1) it has label $3$, and then it is easy to see that the other two edges of the triangle are labeled $2$. Thus $\mathcal{O}$ is arithmetic and is the minimal element in the commensurability class of the Whitehead link (combine \cite[\S 2]{NR2} and \cite[Example 1]{Wie}), and therefore is different from the class of the figure-eight knot complement. But then by \cite{Re1}, there are no knots in this commensurability class of $\mathcal{O}$. This proves (2). 
\end{proof} 




\section{Commensurators containing a reflection} \label{s: contain reflection}

This section is devoted to the proof of Theorem \ref{thm:areflection}. Because of Proposition \ref{P:small} this theorem applies to a hyperbolic AP knot $K$  whose complement's orientable commensurator quotient admits a reflection, but whose full commensurator quotient is not a reflection orbifold. It gives strong restrictions on the topology and combinatorics of $\O_{min}(K)$.

\begin{proof}[Proof of Theorem \ref{thm:areflection}]
By Lemma \ref{l:prelim}(1), $S^3 \setminus K$ admits hidden symmetries and so by \cite[Corollary 4.11]{BBCW1}, $\O_{min}(K)$ has underlying space a ball. Further, $\O_{min}(K)$ has a rigid cusp so its cusp cross section is of the form $S^2(2,3,6), S^2(2,4,4)$ or $S^2(3,3,3)$. 

The full commensurator quotient $\mathcal{O}_{full}(K)$ of $S^3 \setminus K$ is the quotient of $\mathcal{O}_{min}(K)$ by the hypothesized reflection $r: \O_{min}(K) \to \O_{min}(K)$. If each of the three peripheral (i.e. non-compact) edges of the ramification locus of is invariant under $r$, then the cusp cross section of $\O_{min}(K)$ is a reflection orbifold and therefore Lemma \ref{l:reflection} implies that $\mathcal{O}_{full}(K)$ is a reflection orbifold. By Theorem \ref{T:reflections}, $K$ is one of the dodecahedral knots, so $\O_{min}(K)$ is a one-cusped tetrahedral orbifold (cf. \cite[\S 9]{NR1}). Thus $\O_{min}(K)$ satisfies (a). Assume below that this does not happen. Then $r$ leaves exactly one of the peripheral edges of $\Sigma(\O_{min})$ invariant. It follows that cusp cross section of $\O_{min}$ is either $S^2(2,4,4)$ or $S^2(3,3,3)$ (cf. Remark \ref{236 implies reflection}). In the first case $r$ preserves the peripheral edge labeled $2$. 

Denote the truncation of $\O_{min}(K)$ by $\O_{min}(K)^{tr}$ and let $P$ be the intersection of $\O_{min}(K)^{tr}$ with the reflection plane of $r$. Then $|\mathcal{O}_{min}(K)^{tr}|$ is homeomorphic to a $3$-ball and $P$ a properly embedded disc in $|\mathcal{O}_{min}(K)^{tr}|$. By assumption, the circle $\partial P$ contains exactly one of the cone points of the cusp cross section of $\O_{min}(K)$ contained in $\partial \O_{min}(K)^{tr}$. Further, the two remaining cone points both have order $3$ or both have order $4$. 

An open regular $r$-invariant neighborhood of $P \cup \partial \O_{min}(K)^{tr}$ in $\O_{min}(K)^{tr}$ has complement consisting of two connected orbifolds $\mathcal{B}_L$ and $\mathcal{B}_R$. By construction $r(\mathcal{B}_L) = \mathcal{B}_R$. Both $\mathcal{B}_L$ and $\mathcal{B}_R$ have underlying space a $3$-ball. 

Let $\mathcal{S}_L = \partial \mathcal{B}_L$ and observe that $|\Sigma \cap \mathcal{S}_L| \geq 2$. For if $|\Sigma \cap \mathcal{S}_L| = 0$, the singular set of $\O_{min}(K)$ is contained in $P$, contrary to our assumptions, and if $|\Sigma \cap \mathcal{S}_L| = 1$, $\mathcal{S}_L$ would be a bad $2$-suborbifold of $\O_{min}(K)$. 

We claim that $|\Sigma \cap \mathcal{S}_L| \leq 3$. Suppose otherwise. Then $\mathcal{S}_L$ has four or more cone points, at least one of which has order $3$ or $4$, so it is a hyperbolic $2$-orbifold. By hypothesis, $\mathcal{S}_L$ is orbifold-compressible in $\O_{min}(K)$.  Thus there is a compressing orbi-disc $\mathcal{D}$ which meets the singular locus of $\O_{min}(K)^{tr}$ in at most one cone point. Assume that $\mathcal{D}$ is chosen to minimize the number of components of $\mathcal{D} \cap P$ and consider a $2$-suborbifold $\mathcal{I}$ of $\mathcal{D}$ whose boundary is a component $\mathcal{D} \cap P$ which is innermost on $\mathcal{D}$. Then $\mathcal{I} \cup r(\mathcal{I})$ is a $2$-suborbifold of $\O_{min}(K)$ with underlying space $S^2$. Since $\mathcal{D}$, and therefore $\mathcal{I}$, has at most one cone point, $\mathcal{I} \cup r(\mathcal{I})$ has at most two cone points. It cannot have one as otherwise $\O_{min}(K)$ would contain a bad $2$-suborbifold. Thus it has zero or two cone points, and if two, they are cone points of the same order. It follows that $\mathcal{I} \cup r(\mathcal{I})$ is a spherical $2$-suborbifold of $\O_{min}(K)$ and hence must bound an orbi-ball. But then we can reduce the number of components of $\mathcal{D} \cap P$, contradicting our assumptions. Thus $\mathcal{D}$ is disjoint from $P$. 
Since $\mathcal{S}_L$ is one boundary component of a regular neighborhood of $\partial \O_{min}(K)^{tr} \cup P$, removing $\mathcal{B}_L$ from the component of $\O_{min}(K)^{tr} \setminus P$ which contains it results in a product orbifold $\mathcal{P}_L = \mathcal{S}_L \times (0,1)$. Hence any compressing orbi-disc $\mathcal{D}$ for $\mathcal{S}_L$ is contained in $\mathcal{B}_L$. 

Consider a compressing orbi-disc $\mathcal{D}$ for $\mathcal{S}_L$ in $\mathcal{B}_L$. Since $\mathcal{P}_L$ is a product, it contains an annulus $A$ cobounded by $\partial \mathcal{D}$ and a simple closed curve on $P$. Then $\mathcal{D} \cup A \cup r(A) \cup r(\mathcal{D})$ is a $2$-suborbifold with either zero or two cone points.  In either case it must be a spherical $2$-suborbifold bounding an orbi-ball. But then $ \partial \mathcal{D}$ would be inessential in $\mathcal{S}_L$, contrary to our assumptions. We conclude that $\mathcal{S}_L$, and therefore $\mathcal{S}_R = r(\mathcal{S}_L)$, has at most three cone points. Since $\O_{min}(K)$ does not contain an orbifold-incompressible $2$-suborbifold, both $\mathcal{B}_L$ and $\mathcal{B}_R$ must be orbi-balls. 

First consider the case where there are three cone points in $\mathcal{S}_L$. Then the singular loci of the orbi-balls $\mathcal{B}_L$ and of $\mathcal{B}_R$ are tripods. By construction, the endpoints of two edges of each of these tripods lie in $P$ while the endpoints of the third edges are cone points of equal order ($3$ or $4$) on $\partial \O_{min}^{tr}$. The union of the two tripods contains a circular $1$-cycle $a_0$ homeomorphic to a circle and two relative $1$-cycles $a_1$ and $a_2$ homeomorphic to intervals and properly embedded in $\O_{min}(K)$. 

Set $\Sigma = \Sigma(\O_{min}(K)) \cap \O_{min}(K)^{tr}$ and consider $\Sigma \cap P$. There are two  points $x_1, x_2$ in $\Sigma \cap P$ which correspond to the intersection of the legs of the tripods with $P$. Note that each of these two points may or may not be a vertex of $\Sigma(\O_{min}(K))$. If we remove the elements of $\{x_1, x_2\}$ which are not vertices of $\Sigma(\O_{min}(K))$ from $\Sigma$, what remains of $\Sigma \cap P$ inherits the structure of a graph from $\Sigma(\O_{min}(K))$ whose vertices have valency $1$ or $3$. Those of valency $1$ are $x_0$, the unique cone point of $\partial \O_{min}^{tr}$ contained in $\partial P$, and whichever of $x_1, x_2$ is a vertex of $\Sigma(\O_{min}(K))$.

Suppose that there is a circle $b$ in the graph $\Sigma \cap P$. Let $E$ be the interior of the disc in $P$ that it bounds. If $E$ contains both $x_1$ and $x_2$ or neither of them, then it is easy to construct a $2$-sphere in the interior of $|\O_{min}(K)|$ which separates $a_0$ and $b$. If it contains exactly one of $x_1$ and $x_2$, then one of the relative $1$-cycles $a_1, a_2$ can be separated from $b$ by a $2$-sphere. Each of these possibilities contradicts the barrier lemma (Lemma \ref{l:barrier}). Thus $\Sigma \cap P$ is a finite union of trees. Each tree containing an edge has at least two extreme vertices, each extreme vertex has valency $1$, and the vertices of valency $1$ of $\Sigma \cap P$ are contained in $\{x_0, x_1, x_2\}$. Since all other vertices have valency $3$ and the tree $T_0$ containing $x_0$ has at least one edge, it is easy to argue that either 
\begin{itemize}
\item $T_0$ is an interval with boundary $\{x_0, x_1\}$, say, and $\O_{min}(K)$ is a one-cusped tetrahedral orbifold, or
\item 

$T_0$ is a tripod with extreme vertices $\{x_0, x_1, x_2\}$ and  the underlying graph of $\Sigma(\O_{min}(K))$ is as depicted in (b) and (c) of the statement of the theorem. 
\end{itemize}
In the first case, $\O_{min}(K)$ satisfies (a). Suppose that the second case arises and recall that we noted above that the cusp cross section of $\O_{min}(K)$ is $S^2(3,3,3)$ or $S^2(2,4,4)$. The requirements that the interior vertices of $\Sigma(\O_{min}(K))$ correspond to spherical quotients and that $\O_{min}(K)^{tr}$ has no orientation-preserving symmetry allows us to determine the local isometry groups of $\O_{min}(K)$ and we conclude that (b) occurs when the cusp cross section is $S^2(3,3,3)$ and that (c) occurrs when it is $S^2(2,4,4)$. 

Now consider the case when there are two cone points in $\mathcal{S}_L$. Then both $\mathcal{B}_L$ and $\mathcal{B}_R$ are quotients of a ball by a finite cyclic  rotational action, and $\Sigma \cap P$ has one point $x_1$ where $\Sigma$ meets $P$ transversely. The vertices of $\Sigma \cap P$ of valency $1$ include $x_0$ and are contained in $\{x_0, x_1\}$. All other vertices have valency $3$. Let $\Sigma_0$ be the component of $\Sigma \cap P$ containing $x_0$. If $\Sigma_0$ is a tree, it must be an interval (cf. the previous paragraph) with boundary $\{x_0, x_1\}$. Any other component of $\Sigma \cap P$ would contain $1$-cycles, which is easily seen to contradict the barrier lemma. Thus $\Sigma \cap P = \Sigma_0$, so that $\Sigma$ is a tripod with Euclidean labeling, which is impossible. 

Suppose then that $\Sigma_0$ contains circular $1$-cycles. Any other component of $\Sigma \cap P$ has at most one extreme vertex and so must contain circular $1$-cycles as well, which is easily seen to contradict the barrier lemma. Thus $\Sigma \cap P = \Sigma_0$.  

Fix a circular $1$-cycle $b$ in $\Sigma_0$ and observe that if $b$ contained $x_1$, then $x_1$ would be a vertex of $\Sigma(\O_{min}(K))$ of valency at least $4$, which is impossible. It follows from the barrier lemma that the interior of the disc in $P$ bounded by $b$ must contain $x_1$. If there is another circular $1$-cycle $b'$ in $\Sigma_0$, the barrier lemma implies that $b$ and $b'$ have a non-empty intersection, and using the fact that their vertices have valency $3$ in $P$, they share at least one edge. It is then easy to see that $\Sigma_0$ contains a third circular $1$-cycle which does not enclose $x_1$, which contradicts the barrier lemma. Thus $b$ is the only circular $1$-cycle in $\Sigma_0$, and it is easy to argue that $\Sigma_0$ satisfies one of the following two scenarios:
\begin{itemize}
\item $b$ consists of a single edge with end-point $x_2$ and $\Sigma_0$ is the union of $b$ and an edge connecting $x_2$ to $x_0$.  To avoid contradicting the barrier lemma, $x_1$ must be in the interior of the disk bounded by $b$, and will be isolated in $\Sigma \cap P$. 
\item $b$ is the union of two edges and $\Sigma_0$ is the union of $b$ and two other edges - one connecting $x_0$ to one vertex of $b$ and the other connecting $x_1$ to the other vertex of $b$.
\end{itemize}
We can rule out the first possibility since it would imply that $\O_{min}(K)$ admits a rotational symmetry of angle $\pi$. Thus the second possibility holds and therefore the underlying graph of $\Sigma(\O_{min}(K))$ is as depicted in (d). The requirements that the interior vertices of $\Sigma(\O_{min}(K))$ correspond to spherical quotients and that $\O_{min}(K)^{tr}$ has no orientation-preserving symmetry determines the local isometry groups. We conclude that the cusp cross section is $S^2(3,3,3)$ and that the labels on the edges of $\Sigma(\O_{min}(K))$ are as given in (d). This completes the proof. 
\end{proof}

\medskip\noindent

\vspace{.5cm} 
{\scriptsize

\noindent Michel Boileau, Institut de Math\'ematiques de Marseille I2M, UMR 7373, Universit\'e d'Aix-Marseille, Technop\^{o}le Ch\^{a}teau-Gombert, 
39, rue F. Joliot Curie, 13453 Marseille Cedex 13, France 
\newline\noindent
e-mail: michel.boileau@cmi.univ-mrs.fr \\

\noindent
Steven Boyer, D\'ept. de math., UQAM, P. O. Box 8888, Centre-ville, Montr\'eal, Qc, H3C 3P8, Canada
\newline\noindent
e-mail: boyer@math.uqam.ca \\ 

\noindent
Radu Cebanu, Department of Mathematics, Boston College, Chestnut Hill, MA 02467-3806, USA 
\newline\noindent
e-mail: radu.cebanu@gmail.com \\ 

\noindent
Genevieve S. Walsh, Dept. of Math., Tufts University, Medford, MA 02155, USA
\newline\noindent
e-mail: genevieve.walsh@gmail.com

}


\begin{thebibliography}{9999}

\bibitem{Ad} C. Adams, {\it Toroidally alternating knots and links}, Topology {\bf 33} (1994), 353--369.

\bibitem{AR} I. R. Aitchison and J. H. Rubinstein, {\em Combinatorial
cubings, cusps, and the dodecahedral knots}, in
Topology '90, Ohio State Univ. Math. Res. Inst. Publ. {\bf 1}, 
 17-26,  de Gruyter (1992).
 
 \bibitem{AR2} \bysame, {\em Geodesic Surfaces in Knot Complements}, Exp. Math. {\bf 6} (1997), 137--150. 
 
 \bibitem{Banks} J. Banks, {\em The complement of a dodecahedral knot contains an essential closed surface}, http://thales.math.uqam.ca/ banksj/, preprint 2012.
 
 \bibitem{BBCW1} M. Boileau, S. Boyer, R. Cebanu, and G. S. Walsh, {\em Knot commensurability and the Berge conjecture}, Geom. \& Top. {\bf 16} (2012), 625--664. 
 
 \bibitem{BMP} M. Boileau, S. Maillot and J. Porti, \emph{Three-dimensional
orbifolds and their geometric structures}, Panoramas et Synth\`eses,
\textbf{15}, Soci\'et\'e Math\'ematique de France, Paris, 2003.
  
 
\bibitem{computeclosed} B. Burton, A. Coward, and S. Tillmann,  {\em Computing closed essential surfaces in knot complements},  SoCG '13: Proceedings of the Twenty-Ninth Annual Symposium on Computational Geometry, ACM (2013), 405--414.

\bibitem{DOL} I. Dolgachev,  {\em Reflection groups in algebraic geometry}, Bull. Amer. Math. Soc. {\bf 45} (2008), 1--60.
 
 \bibitem{GAW} F. Gonz{\'a}lez-Acu{\~ n}a and W. C. Whitten, {\em Imbeddings of three-manifold groups}, Mem. Amer. Math. Soc.  {\bf 474} (1992).

\bibitem{GHH} O. Goodman, D. Heard, and C. Hodgson, {\em Commensurators of cusped hyperbolic manifolds}, Experiment. Math. {\bf 17} (2008), 283--306.

\bibitem{Neilsmall} N. Hoffman, {\em Small knot complements, exceptional surgeries, and hidden symmetries}, arXiv:1110.3863v2. 

\bibitem{Neildodec} \bysame, {\em On knot complements that decompose into regular ideal dodecahedra}, to appear in Geometriae Dedicata. 

\bibitem{JS} Wm. Jaco and P. B. Shalen, {\em Seifert fibered spaces in 3-manifolds}, Mem. Amer. Math. Soc. {\bf 21}, no. 220, 1979.  

\bibitem{Mar} G. Margulis, {\em Discrete Subgroups of Semi-simple Lie Groups}, Ergeb. der Math. {\bf 17} Springer-Verlag (1989).

\bibitem{Me} Wm. Menasco, {\it Closed incompressible surfaces in alternating knot and link complements}, Top. {\bf 23} (1984), 37--44.

\bibitem{NR1} W. D. Neumann and A. W. Reid, {\em Arithmetic of hyperbolic manifolds}, in
Topology '90, Ohio State Univ. Math. Res. Inst. Publ. {\bf 1},  
273--310,  de Gruyter (1992). 

\bibitem{NR2} \bysame, {\em Notes on Adams' small volume orbifolds}, in
Topology '90, Ohio State Univ. Math. Res. Inst. Publ. {\bf 1},  
273--310,  de Gruyter (1992). 

\bibitem{Oe} U. Oertel, {\it Closed incompressible surfaces in complements of star links}, Pac. J. Math, {\bf 111} (1984), 209--230.

\bibitem{Re1}  A. W. Reid, {\em Arithmeticity of knot complements}, J. London Math. Soc. {\bf 43}
(1991), 171--184.

\bibitem{RW} A. W. Reid and G. S. Walsh, {\em Commensurability classes of two-bridge knot complements},   Alg. \& Geom. Top. {\bf 8} (2008), 1031 -- 1057. 

\bibitem{Sch} R. E. Schwartz, {\em The quasi-isometry classification of rank one lattices}, 
Publ. I.H.E.S. {\bf 82} (1995), 133--168. 

\bibitem{Th}  Wm. Thurston, {\it The Geometry and Topology of 3-manifolds}, Princeton University lecture notes, 1980. Electronic version 1.1: http://www.msri.org/publications/books/gt3m/.

\bibitem{Walshsurvey} G. S. Walsh, {\em Orbifolds and commensurability}, In "Interactions Between Hyperbolic Geometry, Quantum Topology and Number Theory" Contemporary Mathematics {\bf 541} (2011), 221--231.

\bibitem{Wie} N. Wielenberg, {\em The structure of certain subgroups of the Picard group}, Math. Proc. Cam.
Phil. Soc. {\bf 84} (1978), 427--436.


\end{thebibliography}
\end{document}